\documentclass[11pt,letterpaper]{article}
\usepackage{verbatim, amsfonts, amsmath, amssymb, amscd, xcolor, amsthm, graphicx, mathrsfs, setspace, mdwlist, float, calc,floatflt,enumitem,comment}
\newtheorem{thm}{Theorem}[section]
\newtheorem{cor}[thm]{Corollary}

\newtheorem{ques}[thm]{Question}

\theoremstyle{definition}

\newtheorem{proposition}[thm]{Proposition}

\theoremstyle{remark}

\renewcommand{\d }{{\rm d} }

\renewcommand{\H}{\mathcal{H}}
\newcommand{\AH}{\mathcal{AH}}
\newcommand{\AHG}{\mathcal{AH}(G)}
\newcommand{\HG}{\mathcal{H}(G)}
\newcommand{\GG}{\mathcal{G}(G)}

\newcommand{\acts}{\curvearrowright}

\newcommand{\Ga}{\Gamma}

\newcommand{\op}{\operatorname}

\usepackage{titlesec}

\titleformat{\section}
  {\normalfont\scshape}{\thesection}{1em}{}

\begin{document}

\title{Finite extensions of $\H-$ and $\AH$-accessible groups }
\author{S. H Balasubramanya}
\date{}
\maketitle

\begin{abstract}We prove that the group properties of being $\H-$accessible and $\AH$-accessible are preserved under finite extensions. We thus answer an open question from \cite{ABO}. 
\end{abstract}

\maketitle

\section{\bf Introduction}
An important open question related to the class of acylindrically hyperbolic groups is the following (See \cite[Question 2.20]{survey}). 

\begin{ques}\label{mainq} Is the property of being acylindrically hyperbolic preserved under quasi-isometries for finitely generated groups? In other words, if $G$ is a finitely generated acylindrically hyperbolic group, and $H$ is a finitely generated group that is quasi-isometric to $G$, then is $H$ also an acylindrically hyperbolic group ? \end{ques} 

A group $G$ is called \emph{acylindrically hyperbolic} if it admits a non-elementary, acylindrical action on a hyperbolic space. The motivation behind the following question comes from the observation that the class of acylindrically hyperbolic groups serves as a generalization to the classes of non-elementary hyperbolic and relatively hyperbolic groups. The latter two classes are preserved under quasi-isometries (see \cite{gromov}and \cite[Theorem 5.12]{drutu}), making Question \ref{mainq} a natural question to consider. 

An answer to this question seems currently out of reach given the lack of techniques to build an action of $H$ on a hyperbolic space simply starting from an action of $G$ on a (possibly different) hyperbolic space and a quasi-isometry between the two groups. Indeed, this is hard to do even in the case of actions on Cayley graphs that arise from finite generating sets. By \cite[Theorem 1.2]{ah}, if $G$ is acylindrically hyperbolic, then there exists a hyperbolic Cayley graph, say $\Ga(G, A)$, such that $G \acts \Ga(G,A)$ is acylindrical and non-elementary. If $G$ is acylindrically hyperbolic, but not a hyperbolic group, then $A$ is necessarily infinite. Suppose that the groups $G$ and $H$ are quasi-isometric with respect to some finite  generating sets $X$ and $Y$ respectively.  Then the lack of a relationship between $G \acts \Ga(G,X)$ and $G \acts \Ga(G, A)$ makes it unclear as to how we may transfer the desired properties to the group $H$. 

Since the answer to this general question is out of reach, we may consider some special cases of Question \ref{mainq} instead. For instance, we may ask the following.  
 
\begin{ques}\label{simplerq} Let $H <G$ be a finite index subgroup. If $H$ is acylindrically hyperbolic, then is $G$ acylindrically hyperbolic ? 
\end{ques}

While it seems that this question is much simpler, in actuality it is also hard to answer. The primary obstruction here is the lack of any techniques that allow us to build an action of the parent group $G$ starting from the action of a finite index subgroup $H$ while retaining the properties of hyperbolicity and acylindricity. Indeed,  \cite{AHO} and sections of \cite{ABO} explore the possibility of extending group actions in certain settings, especially from peripheral and hyperbolically embedded subgroups, but those methods do not generalize to the finite index case. It is also worth noting, as shown in \cite[Example 1.6]{AHO}, that the extension problem may not be solvable even for finite index subgroups of finitely generated groups. 

However, the authors of \cite{errata} proved that the answer to Question \ref{simplerq} is affirmative in the special case when the finite index subgroup of $G$ is normal and $\AH-$accessible. 

\begin{proposition}\cite[Lemma 6]{errata}\label{wik} Suppose that a group $G$ contains a normal acylindrically hyperbolic subgroup $H$
of finite index which is $\AH$-accessible. Then $G$ is acylindrically hyperbolic.
\end{proposition}

A group $G$ is said to be \emph{$\AH-$accessible} if the poset $\AH(G)$ contains the largest element. (Note that largest elements, if they exist, are unique.) We will show that under the hypothesis of Proposition \ref{wik}, something even stronger holds true : the group $G$ is also $\AH-$accessible. Thus, we have the following. 

\begin{proposition}\label{wtp} Suppose that a group $G$ contains a normal acylindrically hyperbolic subgroup $H$ of finite index which is $\AH$-accessible. Then $G$ is also $\AH-$accessible.
\end{proposition}

Note that the groups $G$ and $H$ in the above proposition are quasi-isometric. In particular, the above result shows that the property of being $\AH-$accessible (and hence acylindrically hyperbolic) is preserved in this very special situation of a quasi-isometry. We will also show that the property of being $\H-$accessible can be transferred under similar hypothesis. 

\begin{proposition}\label{adprop} Suppose that a group $G$ contains a normal, finite index subgroup $H$ such that $H$ is $\H-$accessible. Then $G$ is also $\H-$accessible. \end{proposition}

Obviously, the Propositions \ref{wtp} and \ref{adprop} show that the properties of being $\H-$accessible and $\AH-$accessible are preserved under finite extensions, which answers an open question from \cite{ABO} (See Problem 8.10). \\

One important implication of being $\AH-$accessible is that the group admits a \emph{universal}action. i.e. there is an acylindrical action of the group on a hyperbolic space such that every generalized loxodromic acts loxodromically in this action (see \cite[Section 7]{ABO} for details). Another motivation for understanding whether a given group is $\H-$ or $\AH-$accessible is that when the answer is affirmative, it provides an obvious candidate for which group action is the most informative in order to better understand certain properties of the group. The largest structure is also preserved by every automorphism of the group, which may be useful for studying the group (see \cite[Section 2.3]{ABO} for details).

\section{\bf Preliminaries}

We quickly recall some standard terminology and definitions from \cite{ABO}, before proceeding with the proofs of the results. 

Throughout this paper, all group actions on metric spaces are assumed to be isometric. Given a metric space $S$, we denote by $\d_S$ the distance function on $S$ unless another notation is introduced explicitly. Given a point $s\in S$ or a subset $R\subseteq S$ and an element $g\in G$, we denote by $gs$ (respectively, $gR$) the image of $s$ (respectively $R$) under the action of $g$. Given a group $G \curvearrowright S$ and some $s\in S$, $Gs$ denotes the $G$-orbit of $s$ under the group action.

Let $X$, $Y$ be two generating sets of a group $G$. We say that $X$ is \emph{dominated} by $Y$, written $X\preceq Y$, if the identity map on $G$ induces a Lipschitz map between metric spaces $(G, \d_Y)\to (G, \d_X)$. This is equivalent to the requirement that $\sup_{y\in Y}|y|_X<\infty$, where $|\cdot|_X=\d_X(1, \cdot)$ denotes the word length with respect to $X$.  It is easy to see that $\preceq$ is a preorder on the set of generating sets of $G$ and therefore it induces an equivalence relation in the standard way:
$$
X\sim Y \;\; \Leftrightarrow \;\; X\preceq Y \; {\rm and}\; Y\preceq X.
$$
This is equivalent to the condition that the Cayley graphs $\Ga(G,X)$ and $\Ga(G, Y)$ are $G$-equivariantly quasi-isometric. We denote by $[X]$ the equivalence class of a generating set $X$, and by $\GG$ the set of all equivalence classes of generating sets of $G$. The preorder $\preceq$ induces an order relation $\preccurlyeq $ on $\GG$ by the rule
$$
[X]\preccurlyeq [Y] \;\; \Leftrightarrow \;\; X\preceq Y.
$$

Note that the above order relation is \emph{inclusion reversing}. i.e. if $ Y \subseteq X$, then $[X] \preccurlyeq [Y]$. This definition is consistent with the following observation$\colon$if we take the generating set $X=G$, then the corresponding Cayley graph is a bounded space of diameter $1$. This Cayley graph retains practically no information about the structure of the group $G$, and so $[G]$ is always the smallest element in $\GG$. 

$\HG$ is the partially ordered set of \emph{hyperbolic structures on a group $G$}. It contains equivalence classes of generating sets $[X]$ such that $\Ga(G, X)$ is a hyperbolic space. The partial order on this poset is the one inherited from $\GG$. The subset $\AH(G) \subset \H(G)$ is the partially ordered set of \emph{acylindrically hyperbolic structures on a group $G$}. It contains equivalence classes of generating sets $[X] \in \H(G)$ such that $G \acts \Ga(G, X)$, is acylindrical.

An element $[X] \in \HG$ (resp. $\AHG$)  is called \emph{largest} if for every $[Y] \in \HG$ (resp. every $[Y] \in \AHG$), we have that $[Y] \preccurlyeq [X]$. It is easy to see that if a largest element exists, it is unique. When $\HG$ (resp. $\AH(G)$) contains the largest element, then we say that the group $G$ is $\H-$accessible (resp. $\AH-$accessible). 

Note that while the question of $\H-$accessibility makes sense for any group, the question whether a group is $\AH-$accessible only makes sense in the case of acylindrically hyperbolic groups. Indeed, if the group $G$ is not acylindrically hyperbolic, then the cardinality of $\AH(G)$ is either $1$ or $2$; in either case the group is $\AH-$accessible (see \cite[Theorem 2.6]{ABO}). Further note that for acylindrically hyperbolic groups, the property of being $\AH-$accessible is strictly stronger than being acylindrically hyperbolic. Indeed, every (non-elementary) $\AH-$accessible group is acylindrically hyperbolic, but there exist acylindrically hyperbolic groups that are not $\AH-$accessible (See \cite[Theorem 2.17]{ABO}).  

For further details of this poset, we refer the reader to \cite[Sections 1,7]{ABO}. In particular, we would like to mention \cite[Theorem 2.19]{ABO}, which proves that the following classes of groups are $\AH-$accessible groups$\colon$ mapping class groups of punctured closed surfaces, right-angled Artin groups (RAAGs) and fundamental groups of compact orientable 3-manifolds with empty or toroidal boundary.

\section{\bf Main Results}

We start by proving Proposition \ref{wtp}. The strategy used in the proof of Proposition \ref{wik} in \cite{errata} is the following$\colon$ Given the largest element $[X] \in \AH(H)$ and a finite set $Y$ of distinct representatives of cosets of $H$ in $G$, the authors show that $[X \cup Y] \in \AH(G)$. Since $H \acts \Ga(H,X)$ is non-elementary, it follows that so is $G \acts \Ga(G, X \cup Y)$. We will show that the structure $[X \cup Y]$ is the largest in $\AH(G)$. We will set this notation for the following proof. 

\begin{proof}[Proof of Proposition \ref{wtp}] Let $[Z]$ be any element in $\AH(G)$. Then $\Ga(G,Z)$ is hyperbolic and $G \acts \Ga(G,Z)$ is an acylindrical action. We will show that $[Z] \preccurlyeq [X \cup Y]$, which will prove the result. 

Since $H \leq G$, $H \acts \Ga(G,Z)$ is also acylindrical. Since $H$ has finite index in $G$, the action is also cobounded. Indeed, the $G$-action is cobounded since we are considering an action on a Cayley graph and we have that $\displaystyle \bigcup_{h \in H} hY =G$. Since $Y$ is finite, it is bounded in $d_Z$. By using the \emph{Svarc-Milnor} map from \cite[Lemma 3.11]{ABO}, we get that there exists $[W] \in \AH(H)$ such that $H \acts \Ga(H, W)$ is equivalent to $H \acts \Ga(G, Z).$ (See \cite{AHO} or \cite[Section 3]{ABO} for further details.)

By definition, this means that there is a coarsely $H-$equivariant quasi-isometry $\phi \colon \Ga(G,Z) \to \Ga(H, W)$. Thus there exists a constant $C$ satisfying the following conditions for every $g \in G$ and every $h \in H$.
\begin{equation}\label{eq1} - C + \dfrac{1}{C} d_Z(1, g) \leq d_W(\phi(1), \phi(g)) \leq Cd_Z(1,g) +C  \end{equation}
\begin{equation} \label{eq2}\displaystyle \op{sup}_{h \in H} \hspace{5pt} d_W(h\phi(g) , \phi(hg)) \leq C \end{equation} 

Since $[X]$ is the largest element in $\AH(H)$, we must have that $[W] \preccurlyeq [X]$. Thus there exists a constant $K$ such that  \begin{equation}\label{eq3}\op{sup}_{x \in X} |x|_W \leq K. \end{equation}

Let $x \in X$. Then $x \in H$ and by using (\ref{eq1}) and (\ref{eq2}) we get that $$d_Z(1 ,x) \leq Cd_W( \phi(1), \phi(x)) + C^2 \leq Cd_W(\phi(1), x\phi(1)) + 2C^2.$$
By using the triangle inequality and (\ref{eq3}), we get that  \begin{align*} d_Z(1 ,x) & \leq Cd_W(\phi(1), 1) + Cd_W(1, x) + Cd_W(x, x\phi(1) \big) + 2C^2\\ & \leq 2Cd_W(1, \phi(1)) + CK + 2C^2.\end{align*} Since $d_W(1, \phi(1))$ is a constant independent of the choice of $x$, we get that $$\op{sup}_{x \in X} |x|_Z < \infty. $$ 
Since $Y$ is finite, it follows that  $\underset{s \in X \cup Y} {\op{sup}}$ $|s|_Z < \infty$. Thus $[Z] \preccurlyeq [X \cup Y]$. \end{proof}

\begin{cor}\label{cor} Right-angled coxeter groups (RACG) that contain a right-angled Artin group (RAAG) as a finite index normal subgroup are $\AH-$accessible.
\end{cor}

\begin{proof} By \cite[Theorem 2.19(c)]{ABO}, every RAAG is $\AH$-accessible (where the largest action corresponds to the action on the extension graph). The result now follows from Proposition \ref{wtp}. 
\end{proof} 

Note that by \cite[Lemma 3]{davis}, given any RAAG $H$, there exists a RACG $K$ such that $H \leq K$ is finite index and normal.  Thus by Corollary \ref{cor}, the associated RACG $K$ is $\AH$-accessible. 

Also note that the result that every RACG is $\AH$-accessible is proved in \cite[Theorem A (4)]{ABD}. The proof contained therein uses techniques related to the hierarchically hyperbolic structure of these groups, and shows that the largest action of a RACG corresponds to the action of the group on an altered contact graph. However, in the cases when the RACG contains a RAAG as a finite index normal subgroup, by Corollary \ref{cor}, we may alternatively study the largest action of the RACG by considering the extension of the action of the RAAG on its extension graph (since the two actions are equivalent). 
 
We will now turn our attention to the proof of Proposition \ref{adprop}. The first part of the proof is practically identical to the first part of the proof of \cite[Lemma 6]{errata}, which we reproduce here for the sake of completion. In what follows, we  set the notation that $G$ is a group and $H \leq G$ is a finite index, normal subgroup such that $H$ is $\H-$accessible. 

\begin{proof}[Proof of Proposition \ref{adprop}] Let $[X]$ be the largest element in $\H(H)$, and let $Y$ be a fixed set of distinct representatives of cosets of $H$ in $G$, where the representative of $H$ is $1$. Obviously $Y$ is finite and $X \cup Y$ is a generating set for $G$. 

It follows from Section 2.3 of \cite{ABO} that the action of $Aut(H)$ on $\H(H)$ defined by $\alpha([W]) = [\alpha(W)]$ is a well-defined, order preserving action. Since $[X] \in \H(H)$ is the largest element, we must have $\alpha([X]) = [\alpha(X)]$ for every $\alpha \in Aut(H)$. In particular, since $H$ is normal in $G$ and $Y$ is finite, there exists a constant $L$ such that \begin{equation}\label{eq4}|y^{-1}xy|_X = |y^{-1}x^{-1}y|_X \leq L\end{equation} for all $x \in X$ and for all $y \in Y$. 

We will show that the inclusion map from $\Ga(H, X) \rightarrow \Ga(G, X \cup Y)$ is an $H-$equivariant quasi-isometry.  Obviously this map is $H-$equivariant. To see that the map is coarsely surjective, observe that for every $g \in G$, there exists $h \in H$ and $y \in Y$ such that $g = hy$. Then $$d_{X \cup Y} (g, h) = d_{X \cup Y}(hy, h) =  d_{X \cup Y} (y, 1) \leq 1.$$
Further, we have that for every $h \in H$, $$|h|_{X \cup Y} \leq |h|_X.$$ So it remains to find a constant $M$ such that for every $h \in H$, we have that $|h|_X \leq M |h|_{X \cup Y}$. To this end, let $h \in H$ and let $a_1 a_2....a_n$ be the shortest word in $X^{\pm1} \cup Y^{\pm 1}$ representing $h$. Let $w_0 =1$ and $w_i = a_1a_2...a_i$ for every $i = 1,2,...n$. For every $i$, there exists $y_i \in Y$ such that $w_iy_i^{-1} \in H$. 

Then $$h = (y_0a_1y_1^{-1})(y_1a_2y_2 ^{-1})...(y_{n-2}a_{n-1} y_{n-1} ^{-1})(y_{n-1}a_ny^{-1}_n),$$ where $y_0 =y_n =1$. Observe that \begin{itemize} 
\item[(1)] $y_0a_1 y_1 ^{-1} = w_1 y_1 ^{-1} \in H$
\item[(2)] $y_ja_{j+1}y_{j+1} ^{-1} = (y_j w_j ^{-1})(w_j a_{j+1} y_{j+1} ^{-1}) =  (y_j w_j ^{-1})(w_{j+1}y_{j+1} ^{-1}) \in H$ 
\item[(3)] $y_{n-1}a_ny_n = (y_{n-1} w_{n-1} ^{-1})(w_{n-1}a_n) = (y_{n-1} w_{n-1} ^{-1})w_n \in H$ since $w_n = h \in H$
\end{itemize}

This establishes that $y_{i -1} a_i y_i^{-1} \in H$ for all $i=1,2,...n$. If $a_i \in X$, then $a_i \in H$. By normality, $y_i a_i y_i ^{-1} \in H$. Observe that $$y_{i -1} a_i y_i^{-1} = (y_{i -1}y_i^{-1} ) (y_ia_i y_i^{-1}).$$ Since $y_{i -1} a_i y_i^{-1} \in H$, we must have that $y_{i-1}y_i ^{-1} \in H$, which implies that $y_i H =y_{i-1}H$. Since $Y$ was a collection of distinct coset representatives, this implies that $y_i = y_{i-1}$. It follows from Equation \ref{eq4}, that $|y_{i -1} a_i y_i^{-1}|_X \leq L$.

If $a_i \in Y$, then $|y_{i -1} a_i y_i^{-1}|_Y \leq 3$. Since $Y$ is finite and $y_{i -1} a_i y_i^{-1} \in H$, we can find a constant $N$ (independent of $i$) such that $|y_{i -1} a_i y_i^{-1}|_X \leq N$ in this case as well. Take $M = \op{max}\{L, N\}$. Then $|h|_X \leq M |h|_{X \cup Y}$.  

Since $\Ga(H, X)$ is hyperbolic, so is $\Ga(G, X \cup Y)$. Hence $[X \cup Y] \in \H(G)$. That this is the largest element in $\HG$ follows from an argument almost identical to the one in the proof of Proposition \ref{wtp}, by considering $[Z] \in \HG$ instead (since we do not have to worry about acylindricity in this case). 
\end{proof} 

\bibliographystyle{plain}

\end{document}